\documentclass[12pt]{amsart}

\usepackage{amsmath,amssymb,amsthm,url,bm,microtype,parskip,enumitem,mathtools,tikz-cd,mathrsfs}
\usepackage{xcolor}
\usepackage[colorlinks]{hyperref}

\usetikzlibrary{%
  matrix,%
  calc,%
  arrows%
}

\DeclareMathOperator{\Ker}{ker}

\DeclareMathOperator{\Sp}{Spec}

\DeclareMathOperator{\im}{im}
\DeclareMathOperator{\Hom}{Hom}
\DeclareMathOperator{\Pic}{Pic}

\DeclareMathOperator{\Char}{char}
\DeclareMathOperator{\gl}{gl}
\DeclareMathOperator{\ind}{ind}
\DeclareMathOperator{\Nrd}{Nrd}
\DeclareMathOperator{\GL}{GL}
\DeclareMathOperator{\GO}{GO}
\DeclareMathOperator{\Or}{O}
\DeclareMathOperator{\cd}{cd}
\DeclareMathOperator{\nr}{nr}
\DeclareMathOperator{\ram}{Ram}

\DeclareMathOperator{\Supp}{Supp}
\DeclareMathOperator{\Div}{div}
\DeclareMathOperator{\cont}{cont}
\DeclareMathOperator{\Inf}{Inf}
\DeclareMathOperator{\Ht}{ht}
\DeclareMathOperator{\br}{Br}
\DeclareMathOperator{\Res}{Res}

\newcommand{\comm}[1]{}
\newcommand{\Lrightarrow}{\hbox to1cm{\rightarrowfill}}
\newcommand{\Ldownarrow}{\bigg\downarrow}

\theoremstyle{plain} 
\newtheorem{theorem}{Theorem}[section]
\newtheorem{corollary}[theorem]{Corollary}
\newtheorem{lemma}[theorem]{Lemma}
\newtheorem{prop}[theorem]{Proposition}

\newtheorem{defn}[theorem]{Definition}

\begin{document}
\title[Galois Cohomology of Function Fields]{Galois cohomology of function fields of curves over non-archimedean local fields}
\author{Saurabh Gosavi}
\email{saurabh.gosavi2@gmail.com}
\address{Department of Mathematics \\ Bar-Ilan University \\ Ramat-Gan \\ 5290000 \\ Israel.
}

\maketitle

\begin{abstract}
Let $F$ be the function field of a curve over a non-archimedean local field. Let $m \geq  2$ be an integer coprime to the characteristic of the residue field of the local field. In this article, we show that every element in $H^{3}(F, \mu_{m}^{\otimes 2})$ is of the form $\chi \cup (f) \cup (g)$, where $\chi$ is in $H^{1}(F, \mathbb{Z}/m\mathbb{Z})$ and $(f)$, $(g)$ in $H^{1}(F, \mu_{m})$. This extends a result of Parimala and Suresh (\cite{PS10}), where they show this when $m$ is prime and when $F$ contains $\mu_{m}$.
\end{abstract}

\section{Introduction}
Let $F$ be a field. Recall that when $m$ is coprime to $\Char{F}$, the $m$-torsion subgroup ${}_{m}\br(F)$ of the Brauer group of $F$ can be identified with the Galois cohomology group $H^{2}(F, \mu_{m})$.
When $F$ is a global field or a local field, it is well known that every central simple algebra over $F$ is cyclic. More precisely, one can show that for every $\alpha$ in ${}_{m}\br(F)$, there exists a character $\chi$ in $H^{1}(F, \mathbb{Z}/m\mathbb{Z})$ and an $(f)$ in $H^{1}(F, \mu_{m})$ such that $\alpha = \chi \cup (f)$. 

In this article, we are interested in a ``higher dimensional" analogue of the above classical result. Let $F$ be the function field of a curve over a non-archimedean local field. The degree three Galois cohomology group $H^{3}(F, \mu_{m}^{\otimes 2})$ plays an important role in understanding the arithmetic of such fields. These groups also receive cohomological invariants of principal homogenous spaces under absolutely simple, simply connected linear algerbaic groups, and for some of these algebraic groups, these invariants are known to be classifying (see for eg.\cite[Theorem 5.3]{CPS12}).

A natural question one might ask is whether one can say something about the shape of elements in $H^{3}(F, \mu_{m}^{\otimes 2})$. This has been tackled before by Parimala and Suresh in \cite{PS10}. Let $m$ be a prime not equal to the characteristic of the residue field of the underlying local field. Assume that $F$ contains $\mu_{m}$. They show that (see \cite[Theorem 3.5]{PS10}) every element in $H^{3}(F, \mu_{m})$ is a symbol (cup product of three elements in $H^{1}(F, \mu_{m})$). This fact is then used to show that every quadratic form with at least $9$ variables represents zero non trivially. 

 Their proof however does not seem to generalize to all integers $m$. A crucial ingredient in their proof is a certain local-global principle for divisibility of classes in $H^{3}(F, \mu_{m}^{\otimes 2})$ by those in $H^{2}(F, \mu_{m})$, and some of their technical arguments need $m$ to be a prime.

When $m$ is not necessarily a prime, and without the assumption that $F$ contains $\mu_{m}$, we show that (Corollary \ref{DIVISIBILITY}) this local-global principle holds more generally for semi-global fields (see Definition \ref{SEMI-GLOBAL}) in certain situations. As a consequence of this, we obtain the following theorem:
\begin{theorem}
	\label{MAIN-THM2}
	Let $F$ be the function field of a curve over a non-archimedean local field $K$, and $m$ be a positive integer coprime to the characteristic of the residue field of $K$. For every $\zeta$ in $H^{3}(F, \mu_{m}^{\otimes 2})$, there exists a character $\chi$ in $H^{1}(F, \mathbb{Z}/m\mathbb{Z})$, and classes $(f)$ and $(g)$ in $H^{1}(F, \mu_{m})$ such that $\zeta = \chi \cup (f) \cup (g)$.
\end{theorem}

The strategy of our proof of the local-global principle for divisibility (Corollary \ref{DIVISIBILITY}) is different from that adopted in \cite{PS10} or \cite{PS15}. We prove it using the field patching technique of Harbater, Hartmann and Krashen. This technique has been used to resolve a number of arithmetic problems over semi-global fields. The fields we are interested in, i.e., function fields of curves over non-archimedean local fields are examples of semi-global fields. 

In Proposition \ref{LGP-WRT-PATCHES1}, we show that the local-global principle for divisibility holds with respect to overfields of $F$ coming from a normal model of $F$. Using this key observation and standard techniques, we obtain a local-global principle with respect to discrete valuations in Theorem \ref{LGP-WRT-DVR} and its Corollary \ref{DIVISIBILITY}. Along the way, we also show that a local-global principle for divisibility by symbols also holds for higher cohomology classes with $\mu_{2}$ coefficients. Adapting the proof of \cite[Theorem 5.2]{PS15} to our situation then enables us to obtain Theorem \ref{MAIN-THM2}. 
\section{Acknowledgements}
I thank Prof. Daniel Krashen, Prof. Eliyahu Matzri and  Prof. V. Suresh for enlightening conversations and feedback on the writing. This research was supported by the Israel Science Foundation (grant no. 630/17).

\section{Preliminaries}
\subsection{Galois Cohomology}
\label{GALOIS-COHOMOLOGY}
Let $F$ be a field and $m \geq 2$ an integer coprime to $\Char{F}$. Let $\mu_{m}$ be the group of $m^{th}$ roots of unity. For $j \geq 0$, the group $H^{i}(F, \mu_{m}^{\otimes j})$ is the degree $i$ Galois cohomology groups of $F$ with coefficients in $\mu_{m}^{\otimes j}$. Recall that by Kummer theory, one has the isomorphism $H^{1}(F, \mu_{m}) \cong F^{\times}/(F^{\times})^{m}$. Every class in $H^{1}(F, \mu_{m})$ is of the form $\sigma \mapsto \sigma(\sqrt[m]{f})/f$. Under the isomorphism, this class maps to the $m^{th}$-power class $(f)$ in $F^{\times}/(F^{\times})^{m}$. We will make no destinction between the classes $\sigma \mapsto \sigma(\sqrt[m]{f})/f$ and $(f)$. 

The group $H^{1}(F, \mathbb{Z}/m\mathbb{Z})$ where the absolute Galois group $\Gamma_{F}$ of $F$ acts trivially is the group of characters of the absolute Galois group, i.e., $H^{1}(F, \mathbb{Z}/m\mathbb{Z}) \cong\Hom_{\cont}(\Gamma_{F},\mathbb{Z}/m\mathbb{Z})$. 

Observe that when $\mu_{m} \subset F$, for every $j$, we have the following isomorphisms of Galois modules $\mu_{m}^{\otimes j} \cong \mathbb{Z}/m\mathbb{Z}$. For any field extension $L/F$, we have the restriction map
\[ \Res_{L/F}: H^{i}(F, \mu_{m}^{\otimes j}) \rightarrow H^{i}(L, \mu_{m}^{\otimes j}).\]
We will denote the image of a class $\alpha$ under $\Res_{L/F}$ by $\alpha \otimes L$.

Let $v$ be a discrete valuation on $F$ with residue field $k(v)$. Let $F_{v}$ denote the completion of $F$ at $v$. We assume that $\Char(k(v))$ is coprime to $m$. In this setting, we have the following residue homomorphism (see \cite[Chapter II]{GMS04} for the construction):
\[ \partial_{v}: H^{i}(F, \mu_{m}^{\otimes j}) \rightarrow H^{i-1}(k(v), \mu_{m}^{\otimes j-1}).  \]
Note that $\partial_{v}$ factors through the restriction map $H^{i}(F, \mu_{m}^{\otimes j}) \rightarrow H^{i}(F_{v}, \mu_{m}^{\otimes j})$. 
For the complete discretely valued field $F_{v}$, one has the following exact sequence (see \cite[Proposition 7.7]{GMS04}):

\begin{equation} 
	\label{WITT-EXACT-SEQUENCE}
	0 \rightarrow H^{i}(k(v), \mu_{m}^{\otimes j}) \overset{\Inf}\longrightarrow H^{i}(F_{v}, \mu_{m}^{\otimes j}) \overset{\partial_{v}}{\longrightarrow} H^{i-1}(k(v), \mu_{m}^{\otimes j-1}) \rightarrow 0. 
\end{equation}

Once we fix a parameter $\pi_{v}$ of $F_{v}$, we have a splitting $s_{\pi_{v}}: H^{i-1}(k(v), \mu_{m}^{\otimes j-1}) \rightarrow H^{i}(F_{v}, \mu_{m}^{\otimes j})$ of $\partial_{v}$ defined as $s_{\pi_{v}}(\beta) = \Inf(\beta)\cup(\pi_{v})$. Thus we obtain the following noncanonical decomposition:
\begin{equation} 
\label{SPLIT-DECOMP-WITT}
	H^{i}(F_{v}, \mu_{m}^{\otimes j}) \cong H^{i}(k(v), \mu_{m}^{\otimes j}) \oplus H^{i-1}(k(v), \mu_{m}^{\otimes j-1}).
\end{equation}
We will sometimes refer to this decomposition as the \emph{Witt Decomposition} and to the exact sequence \eqref{WITT-EXACT-SEQUENCE} as the \emph{Witt Exact Sequence}.

Let $\alpha$ be a class in $H^{i}(F_{v}, \mu_{m}^{\otimes j})$. Consider $\alpha_{1} := \Inf(\partial_{v}(\alpha))$, the image of $\partial_{v}(\alpha)$ under $\Inf: H^{i-1}(k(v), \mu_{m}^{\otimes j-1}) \rightarrow H^{i-1}(F_{v}, \mu_{m}^{\otimes j-1})$.
Once we fix a parameter $\pi_{v}$, for every $\alpha$ in $H^{i}(F_{v}, \mu_{m}^{\otimes j})$, by the isomorphism \eqref{SPLIT-DECOMP-WITT}, we can write 
\[
\alpha = \alpha_{0} + \alpha_{1} \cup (\pi_{v}),
\] for some $\alpha_{0}$ in the image of the inflation map $\Inf: H^{i}(k(v), \mu_{m}^{\otimes j}) \rightarrow H^{i}(F_{v}, \mu_{m}^{\otimes j})$.

Now suppose that $F$ is the function field of a normal scheme $\mathscr{X}$. Each codimension one point $x$ defines a discrete valuation $v_{x}$ of $F$. Assume that the characteristic of the residue field $k(x)$ of each codimension one point $x$ is coprime to $m$. Let $\zeta$ be a class in $H^{i}(F, \mu_{m}^{\otimes j})$. Recall that the \textbf{ramification divisor} $\ram_{\mathscr{X}}(\zeta)$ of $\zeta$ is defined as the sum of the prime divisors $\overline{\{x\}}$ such that $\partial_{v_{x}}(\zeta) \neq 0$.
\subsection{Field Patching}
\label{FIELD-PATCHING}
The field patching technique was introduced by Harbater and Hartmann in \cite{HH10} and was subsequently refined in a series of papers by Harbater, Hartmann and Krashen with a host of applications. In \cite{HHK15(1)}, the authors provide a conceptual framework for local-global principles to hold for principal homogenous spaces under connected, rational linear algebraic groups over the so called semi-global fields. The authors in \cite{HHK15(2)} also consider local-global principles for Galois cohomology groups.

\begin{defn}
	\label{SEMI-GLOBAL}
	Let $K$ be a complete discretely valued and $X/K$ be a geometrically integral curve. The function field $F$ of $X/K$ is called a semi-global field. 
\end{defn}
Examples of semi-global field include function fields of curves over non-archimedean local fields, the fields which are of interest to us in this article.

Let $R$ be a complete discretely valued field with fraction field $K$ and residue field $k$. Let $X/K$ be a geometrically integral curve and $F$ be the function field of $X/K$. Note that there always exists a normal projective model $\mathscr{X}/\Sp{R}$ of $F$. Thus $\mathscr{X}/\Sp{R}$ is a flat projective curve of relative dimension one. The fiber $\mathscr{X}_{k}$ over the maximal ideal of $\Sp{R}$ is called the special fiber or the closed fiber.

We will briefly introduce the field-patching set-up. We refer the reader to \cite{HH10} and \cite{HHK09} for details. Let $\mathcal{P}$ be a finite set of closed points in $\mathscr{X}_{k}$, which includes the points where distinct irreducible components of $\mathscr{X}_{k}$ meet. We also assume that $\mathcal{P}$ contains at least one point on every irreducible component of $\mathscr{X}_{k}$. Let $\mathcal{U}$ be the set of irreducible components of $\mathscr{X}_{k} \setminus \mathcal{P}$. For each $\xi$ in $\mathcal{P} \cup \mathcal{U}$, we associate overfields $F_{\xi}$ of $F$. We will define these fields now.

Let $U$ be in $\mathcal{U}$. Let $R_{U}$ be the subring of $F$ consisting of rational functions that are regular on $U$. Any parameter $t$ of the underlying complete discretely valued ring $R$ is regular on $U$ and thus $t$ lies in $R_{U}$. Let $\widehat{R_{U}}$ be the $t$-adic completion of $R_{U}$. Since $R_{U}$ is an excellent normal domain, so is $\widehat{R_{U}}$. Let $F_{U}$ be the fraction field of $\widehat{R_{U}}$.

For $P$ in $\mathcal{P}$, let $R_{P}$ be the local ring of $\mathscr{X}$ at $P$ with maximal ideal $\mathfrak{m}_{P}$. Let $\widehat{R_{P}}$ be the $\mathfrak{m}_{P}$-adic completion of $R_{P}$. Again, $\widehat{R_{P}}$ is a noraml domain since $R_{P}$ is excellent. Let $F_{P}$ be the fraction field of $\widehat{R_{P}}$.

 Let $P$ be in $\mathcal{P}$. Let $\wp$ be a height one prime ideal of $\widehat{R_{P}}$ containing the parameter $t$ of $R$. Note that the contraction of $\wp$ to the local ring $R_{P}$ determines an irreducible component of $\mathscr{X}_{k}$ in $\mathcal{U}$. Such primes $\wp$ will be called \textbf{branches} at $P$ lying on $U$. Note that the localization $(\widehat{R_{P}})_{\wp}$ is a discrete valuation ring. Let $\widehat{R_{\wp}}$ be the completion of this discrete valaution ring with respect to its maximal ideal. Let $F_{\wp}$ be the fraction field of $\widehat{R_{\wp}}$. Note that $F_{\wp}$ is a complete discretely valued field.
It turns out that for all triples $(U, P, \wp)$ with $P \in \overline{U}$ and $\wp$ a branch incident at $P$ and lying on $U$, the field $F_{\wp}$ is an overfield of both $F_{U}$ and $F_{P}$. We denote the set of all branches associated to all pairs $(U, P)$ with $P$ in $\overline{U}$ by $\mathcal{B}$.

The field patching technique enables us to prove two kinds of statements for certain algebraic structures (central simple algebras, quadratic forms etc.) and Galois cohomology classes. First, we can show that whenever there are algebraic structures $\alpha_{U}$ (or Galois cohomology classes) defined over each $F_{U}$ for $U$ in $\mathcal{U}$ and $\alpha_{P}$ over each $F_{P}$ in $\mathcal{P}$ which are compatible in the sense that for each branch $\wp$ incident at $P$ and lying over $U$, we have $\alpha_{U} \otimes_{F_{U}}F_{\wp} = \alpha_{P} \otimes_{F_{P}} F_{\wp}$, then they give rise to a global structure $\alpha$ over $F$ inducing $\alpha_{U}$ and $\alpha_{P}$. Second, one can also show that sometimes, this $\alpha$ is unique. These statements for Galois cohomology classes are pithily encapsulated in the following ``Mayer-Vietoris sequence" (see \cite[Theorem 3.1.5, Corollary 3.1.6]{HHK15(2)}):

Let $m > 1$ be an integer coprime to the characteristic of the residue field $k$ of the semi-global field $F$. Then for $n > 1$, we have the following exact sequence:
\[ 0 \rightarrow H^{n}(F, \mu_{m}^{\otimes n-1}) \rightarrow \prod_{\xi \in \mathcal{P}\cup \mathcal{U}}H^{n}(F_{\xi}, \mu_{m}^{\otimes n-1}) \rightarrow \prod_{\wp \in \mathcal{B}}H^{n}(F_{\wp}, \mu_{m}^{\otimes n-1}) \rightarrow 0.  \]

We will use this sequence to obtain a local global principle for divisibility with respect to the overfields $F_{\xi}$. (see Propositions \ref{LGP-WRT-PATCHES1} and \ref{LGP-WRT-PATCHES2}).

\section{Local-Global principle with respect to patches.}

Let $\alpha$ be in $\br(F)$. Let $A/F$ be a central simple algebra in the class of $\alpha$. Recall that if $B/F$ is another central simple algebra Brauer equivalent to $A$, then $\Nrd(A^{\times}) = \Nrd(B^{\times})$. We will denote this group by $\Nrd(\alpha)$.

We need the following terminology which we borrow from \cite{HW20}:

For $m \geq 1$, we say that a field $F$ is Rost $m$-divisible if for every $\alpha$ in ${}_{m}\br(F)$ the following holds:

For $x$ in $F^{\times}$,
\[ \alpha \cup (x) = 0 \iff x \in \prod_{d|m}[\Nrd(d\alpha)]^{d}. \]

The authors in \cite{HW20} show that if $F$ is a Henselian discretely valued field with residue field $k$ then in some cases $F$ is Rost $m$-divisible when $m$ is a prime not equal to $\Char{k}$.

From now on, we will fix some notation. Unless stated otherwise, $F$ will denote a semi-global field, i.e., $F$ is the function field of a curve over a complete discretely valued field $K$. The residue field of $K$ will be denoted by $k$. We will sometimes abusively say that $k$ is the residue field of the semi-global field $F$. 
\begin{prop}
\label{LGP-WRT-PATCHES1}
Let $F$ be a semi-global field, and let $\alpha$ be an element in $H^{2}(F, \mu_{m})$ and $\zeta$ be in $H^{3}(F, \mu_{m}^{\otimes 2})$ for some $m > 1$ coprime to $\Char(k)$. Suppose that either every branch field $F_{\wp}$ is a Rost $m$-divisible field for every $\wp$ in $\mathcal{B}$, or $\ind(\alpha)$ is square-free. If for every $\xi$ in $\mathcal{P} \cup \mathcal{U}$, there exist elements $(f_{\xi})$ in $H^{1}(F_{\xi}, \mu_{m})$ such that $\zeta = \alpha \cup (f_{\xi})$, then there exists an element $f$ in $F^{\times}$ such that $\zeta = \alpha \cup (f)$.
\end{prop}
\begin{proof}

Consider the diagram obtained by the two Meyer-Vietoris sequences (see \cite[Theorem 3.1.3]{HHK15(2)}):

\begin{equation}
\label{MEYER-VIETORIS}
  \setlength{\arraycolsep}{1pt}
  \begin{array}{*{9}c}
     & & H^1(F, \mu_{m}) & \Lrightarrow & \prod_{\xi \in \mathcal{P} \cup \mathcal{U}}H^1(F_{\xi}, \mu_{m}) & \Lrightarrow & \prod_{\wp \in \mathcal{B}} H^1(F_{\wp}, \mu_{m}) & \Lrightarrow & 0\\
    & & \Ldownarrow{\alpha \cup {\cdot} } & & \Ldownarrow{\prod_{\xi}\alpha \cup \cdot} & & \Ldownarrow{\prod_{\wp}\alpha \cup \cdot} & & \\
    0 &\Lrightarrow & H^3(F, \mu_{m}^{\otimes 2} ) & \Lrightarrow & \prod_{\xi \in \mathcal{P} \cup \mathcal{U}} H^3(F_{\xi}, \mu_{m}^{\otimes 2}) & \Lrightarrow & \prod_{\wp \in \mathcal{B}} H^3(F_{\wp}, \mu_{m}^{\otimes 2}) & \Lrightarrow & 0. 
  \end{array}
\end{equation}

By the Snake Lemma, we obtain the following exact sequence:

\begin{align} 
\label{FUNDAMENTAL-SEQUENCE}
\Ker(\prod_{\xi \in \mathcal{P} \cup \mathcal{U}} \alpha \cup \cdot) \overset{\phi}\longrightarrow \Ker(\prod_{\wp \in \mathcal{B}}  \alpha \cup \cdot) \overset{\delta}{\longrightarrow} \frac{H^3(F, \mu_{m}^{\otimes 2} )}{\im(\alpha \cup \cdot)} \overset{\gl}{\longrightarrow} \prod_{\xi \in \mathcal{P} \cup \mathcal{U}} \frac{H^3(F_{\xi}, \mu_{m}^{\otimes 2})}{\im(\alpha \cup \cdot)}. 
\end{align}

Here $\delta$ is the connecting map. To prove our claim, we need to show that $\ker(\gl) = 0$. To establish this, it suffices to show that $\phi$ is surjective. 

For every $\wp$ in $\mathcal{B}$, let $(x_{\wp})$ be an element in $H^1(F_{\wp}, \mu_{m})$ such that $\alpha \cup (x_{\wp}) = 0$. For $d \geq 1$, let $D_{d\alpha}$ be the underlying division algebra of the Brauer class $d\alpha$. 

We will first assume that $F_{\wp}$ is Rost $m$-divisible. For every $d$ dividing $m$, there exist elements $\theta_{d, \wp}$ in $(D_{d\alpha}\otimes_{F} F_{\wp})^{\times}$ such that 
\[  x_{\wp} = \prod_{d | m} [\Nrd(\theta_{d, \wp})]^{d}{y_{\wp}}^{m},  \]
where $y_{\wp}$ is some element in $F_{\wp}^{\times}$. Since the group $\GL_{1}(D_{d\alpha})$ is a rational, by \cite[Theorem 3.6]{HHK09}, there exist elements $\theta_{d, U}$ in $(D_{d\alpha} \otimes_{F} F_{U})^{\times}$ and $\theta_{d, P}$ in $(D_{d\alpha} \otimes_{F} F_{P})^{\times}$ such that $\theta_{d, \wp} = \theta_{d, U} \theta_{d, P}^{-1}$ for every triple $(U, P, \wp)$. Also for every such triple, there exist elements $y_{U}$ in $F_{U}^{\times}$ and $y_{P}$ in $F_{P}^{\times}$ such that $y_{\wp} = y_{U}y_{P}^{-1}$ because $\mathbb{G}_{m}$ is rational. Consider the elements $x_{U}$ and $x_{P}$ respectively in $F_{U}^{\times}$ and $F_{P}^{\times}$:
\begin{align*} 
x_{U}:=& \prod_{d | m} [\Nrd(\theta_{d, U})]^{d}y_{U}^{m}\\
x_{P}:=& \prod_{d | m} [\Nrd(\theta_{d, P})]^{d}y_{P}^{m} .
\end{align*}
Because the reduced norm is multiplicative, we have $x_{\wp} = x_{U}x_{P}^{-1}$. Since $d\alpha \cup (\Nrd(\theta_{d, U})) = 0$, it follows that $\alpha \cup (x_{U}) = 0$. Similarly, $\alpha \cup (x_{P}) = 0$. This shows that the map $\phi$ in the exact sequence \eqref{FUNDAMENTAL-SEQUENCE} is surjective.

Finally we assume that $\ind(\alpha)$ is square-free. Since $\alpha \cup (x_{\wp}) = 0$, by \cite[Theorem 12.2]{MS82}, there exists an element $\theta_{\wp}$ in $(D \otimes_{F} F_{\wp})^{\times}$ such that $x_{\wp} = \Nrd(\theta_{\wp})$. Again by the rationality of $\GL_{1}(D_{\alpha})$, there exist elements $\theta_{U}$ and $\theta_{P}$ respectively in $(D \otimes_{F} F_{U})^{\times}$ and $(D \otimes_{F} F_{P})^{\times}$ such that $\theta_{\wp} = \theta_{U} \theta_{P}^{-1}$. Consider the elements $x_{U} : = \Nrd(\theta_{U})$ and $x_{P}: = \Nrd(\theta_{P})$. Note that, we have $x_{\wp} = x_{U}x_{P}^{-1}$, $\alpha \cup (x_{U}) = 0$ and $\alpha \cup (x_{P}) = 0$. Thus $\phi$ in the exact sequence \eqref{FUNDAMENTAL-SEQUENCE} is surjective.   
\end{proof}

One can also obtain a local to global principle for divisibility of Galois cohomology classes in $H^n$ with $\mu_2$ coefficients by those in $H^{n-1}$. We will need the Milnor conjecture (proved by Orlov-Vishik-Voevodsky in \cite{OVV07}) for this.

\begin{prop}
\label{LGP-WRT-PATCHES2}
For $n \geq 2$, let $\zeta$ be a class in $H^n(F, \mu_2)$ and $\alpha$ be a symbol in $H^{n-1}(F, \mu_2)$. If for every $\xi$ in $\mathcal{P} \cup \mathcal{U}$, we have $\zeta = \alpha \cup (f_{\xi})$ on $F_{\xi}$, then there exists an $f$ in $F^{\times}$ such that $\zeta = \alpha \cup (f)$ on $F$.
\end{prop}  

\begin{proof}
Consider the following diagram similar to the one in \eqref{MEYER-VIETORIS}
\begin{equation}
\label{MEYER-VIETORIS1}
  \setlength{\arraycolsep}{1pt}
  \begin{array}{*{9}c}
     & & H^{1}(F, \mu_{2}) & \Lrightarrow & \prod_{\xi \in \mathcal{P} \cup \mathcal{U}}H^{1}(F_{\xi}, \mu_{2}) & \Lrightarrow & \prod_{\wp \in \mathcal{B}} H^{1}(F_{\wp}, \mu_{2}) & \Lrightarrow & 0\\
    & & \Ldownarrow{\alpha \cup {\cdot} } & & \Ldownarrow{\prod_{\xi}\alpha \cup \cdot} & & \Ldownarrow{\prod_{\wp}\alpha \cup \cdot} & & \\
    0 &\Lrightarrow & H^n(F, \mu_{2}) & \Lrightarrow & \prod_{\xi \in \mathcal{P} \cup \mathcal{U}} H^n(F_{\xi}, \mu_{2}) & \Lrightarrow & \prod_{\wp \in \mathcal{B}} H^n(F_{\wp}, \mu_{2}). 
  \end{array}
\end{equation}
By the Snake Lemma, we obtain the following exact sequence:

\begin{align} 
\label{FUNDAMENTAL-SEQUENCE1}
\Ker( \prod_{\xi \in \mathcal{P} \cup \mathcal{U}} \alpha \cup \cdot ) \overset{\phi}\longrightarrow \Ker( \prod_{\wp \in \mathcal{B}}  \alpha \cup \cdot ) \overset{\delta}{\longrightarrow} \frac{H^n(F, \mu_{2})}{\im(\alpha \cup \cdot)} \overset{\gl}{\longrightarrow} \prod_{\xi \in \mathcal{P} \cup \mathcal{U}} \frac{H^n(F_{\xi}, \mu_{2})}{\im(\alpha \cup \cdot)}. 
\end{align}

As in the proof of Proposition \ref{LGP-WRT-PATCHES1}, it suffices to show that $\phi$ is surjective. 

For every $\wp$ in $\mathcal{B}$, let $(x_{\wp})$ be an element in $H^1(F_{\wp}, \mu_2)$ such that $\alpha \cup (x_{\wp}) = 0$. Let $q_{\alpha}$ be an $(n-1)$-fold Pfister form in the class of a lift of $\alpha$ under the surjective map $I^{n-1}(F_{\wp}) \rightarrow H^{n-1}(F_{\wp}, \mu_2)$. Since $\alpha \cup (x_{\wp}) = 0$, and the kernel of the map $I^{n}(F_{\wp}) \rightarrow H^{n}(F_{\wp}, \mu_{2})$ is $I^{n+1}(F_{\wp})$, the class of the form $q_{\alpha} \otimes \langle 1, -x_{\wp} \rangle$ lies in $I^{n+1}(F_{\wp})$. But since the dimension of $q_{\alpha}\otimes \langle 1, -x_{\wp} \rangle$ is $2^{n}$, by the Arason-Pfister Hauptsatz (see \cite[Chapter X, Theorem 5.1]{Lm05}), we have $q_{\alpha} \cong x_{\wp}q_{\alpha}$. After composing the isometry by a reflection in the orthogonal group $\Or_{q_{\alpha}}$ if necessary, we see that $x_{\wp}$ is in the image of the multiplier map $\mu: \GO^{+}_{q_{\alpha}} \rightarrow \mathbb{G}_{m}$. Therefore there exists $g_{\wp}$ in $\GO_{q_{\alpha}}^{+}(F_{\wp})$ such that $\mu(g_{\wp}) = x_{\wp}$. 

By \cite[Proposition 7]{Me96}, the group $\GO_{q_{\alpha}}^{+}$ is connected and stably rational, and therefore also connected and retract rational. Thus by \cite[Theorem 5.1.1]{K10}, $\GO_{q_{\alpha}}^{+}$ satisfies the simultaneous factorization property, i.e., there exist elements $g_{U}$ and $g_{P}$ respectively in $\GO_{q_{\alpha}}^{+}(F_{U})$ and $\GO_{q_{\alpha}}^{+}(F_{P})$ such that $g_{\wp} = g_{U}g_{P}^{-1}$ for every triple $(U, P, \wp)$. Consider the elements $x_{U}: = \mu(g_{U})$ and $x_{P}:= \mu(g_{P})$. One therefore has that $x_{\wp} = x_{U}x_{P}^{-1}$. It remains to show that $\alpha \cup (x_{U}) = 0$ and $\alpha \cup (x_{P}) = 0$.

Since $x_{U}$ and $x_{P}$ are similitudes of $q_{\alpha}$ over $F_{U}$ and $F_{P}$, and since $q_{\alpha}$ is a Pfister form, $x_{U}$ and $x_{P}$ are values of $q_{\alpha} \otimes F_{U}$ and $q_{\alpha} \otimes F_{P}$. This implies that in the Witt rings $W(F_{U})$ and $W(F_{P})$, we have $q_{\alpha} \langle 1, -x_{U} \rangle = 0$ and $q_{\alpha} \langle 1, -x_{P} \rangle = 0$. Therefore $\alpha \cup (x_{U}) = 0$ and $\alpha \cup (x_{P}) = 0$. This shows that $\phi$ is surjective.
\end{proof}

We will need the following theorem, Gersten's conjecture for \'etale cohomology for Henselian two dimensional local rings, proved in \cite[Theorem 2.7]{S20}.
\begin{theorem}[Sakagaito]
	\label{SAKAGAITO}
	Let $R$ be a complete two dimensional regular local ring with fraction field $F$ and residue field $k$. Assume that $m$ is coprime to $\Char{k}$. We have the following exact sequence
	\[ 0 \rightarrow H^{i}_{\acute{e}t}(R, \mu_{m}^{\otimes j}) \rightarrow H^{i}(F, \mu_{m}^{\otimes j}) \overset{\oplus \partial_{\mathfrak{p}}}{\longrightarrow} \bigoplus_{\Ht(\mathfrak{p}) = 1} H^{i-1}(k(\mathfrak{p}), \mu_{m}^{\otimes j-1}) \rightarrow H^{i-2}(k, \mu_{m}^{\otimes j-2}) \rightarrow 0.  \]
\end{theorem}

\begin{prop}
\label{TWO-DIM-PROP}
Let $R$ be a complete two dimensional regular local ring with fraction field $F$, residue field $k$, and maximal ideal $\mathfrak{m} = (\pi, \delta)$. For $i \geq 2, \ j \geq 1$ and $m$ coprime to $\Char(k)$, let $\zeta$ be a class in $H^{i}(F, \mu_{m}^{\otimes j})$ and $\alpha$ be a class in $H^{i-1}(F, \mu_{m}^{\otimes j-1})$. Suppose that the ramification of $\zeta$ and $\alpha$ is supported along $\{ \pi, \delta \}$. If $\zeta = \alpha \cup (f_{\pi})$ in $H^{i}(F_{\pi}, \mu_{m}^{\otimes j})$, where $F_{\pi}$ is the completion of $F$ with respect to $\pi$, and $f_{\pi}$ is some element in $F_{\pi}^{\times}$, then there exists an $f$ in $F^{\times}$ such that $\zeta = \alpha \cup (f)$.
\end{prop}

\begin{proof}
 We may assume that $f_{\pi}$ is an element of the ring of integers $R_{\pi}$ of $F_{\pi}$. Therefore $f_{\pi} = u \pi^r$ where $u$ is a unit in $R_{\pi}$, and $r \geq 0$ some integer. Let $k(\pi)$ denote the residue field of $R_{\pi}$. Note that $k(\pi)$ is a complete discretely valued field with parameter $\overline{\delta}$ and residue field $k$. We therefore have $\overline{u} = \overline{w} \overline{\delta}^{s}$, where $\overline{w}$ is a unit in the ring of integers $k[\pi]$ of $k(\pi)$. Let $\overline{\overline{w}}$ be the reduction of $\overline{w}$ to the residue field $k$ of $k[\pi]$, and let $w$ be the lift of $\overline{\overline{w}}$ to $F$. Since $m$ is coprime to $\Char(k)$, by Hensel's Lemma, we have $(u) = (w \delta^{s})$ in $H^1(F_{\pi}, \mu_{m})$. Thus in $H^1(F_{\pi}, \mu_{m})$, we obtain that $(f_{\pi}) = (w \pi^r \delta^s)$. Set $f := w \pi^r \delta^s$. Note that this is an element in $F^{\times}$.
 
Now consider the element $\widetilde{\zeta} := \zeta - \alpha \cup (f)$. Note that since $\zeta$, $\alpha$ and $(f)$ are ramified along $\{\pi, \delta\}$, it follows that $\widetilde{\zeta}$ is ramified along $\{ \pi, \delta \}$ and $\widetilde{\zeta} \otimes F_{\pi} = 0$ in $H^{i}(F_{\pi}, \mu_{m}^{\otimes j})$. We claim that $\widetilde{\zeta} = 0$. 

Since the residue map $\partial_{\pi}: H^{i}(F, \mu_{m}^{\otimes j}) \rightarrow H^{i-1}(k(\pi), \mu_{m}^{\otimes j-1})$ factors through $H^{i}(F_{\pi}, \mu_{m}^{\otimes j})$, we get that $\partial_{\pi}(\widetilde{\zeta}) = 0$. By Theorem \ref{SAKAGAITO}, we have $\partial_{\overline{\pi}}(\partial_{\delta}(\widetilde{\zeta})) = - \partial_{\overline{\delta}}(\partial_{\pi}(\widetilde{\zeta})) = 0$, i.e., $\partial_{\delta}(\widetilde{\zeta})$ maps to zero under the residue map $\partial_{\overline{\pi}}: H^{i -1}(k(\delta), \mu_{m}^{\otimes j - 1}) \rightarrow H^{i-2}(k, \mu_{m}^{\otimes j - 2})$. Thus, by the exact sequence \eqref{WITT-EXACT-SEQUENCE}, $\partial_{\delta}(\widetilde{\zeta})$ comes from $H^{i-1}(k, \mu_{m}^{\otimes j}) \cong H^{i-1}_{\acute{e}t}(R, \mu_{m}^{\otimes j})$. This isomorphism follow from the proper base change theorem (see \cite[Chapter VI, Corollary 2.7]{Mi80}). We abuse notation and consider $\beta : = \partial_{\delta}(\widetilde{\zeta})$ to be an element in $H^{i-1}_{\acute{e}t}(R, \mu_{m}^{\otimes j-1}) \hookrightarrow H^{i-1}(F, \mu_{m}^{\otimes j-1})$. Now notice that $\partial_{\delta}(\widetilde{\zeta} - \beta \cup (\delta)) = 0$. Note that since $\beta$ comes from $H^{i-1}_{\acute{e}t}(R, \mu_{m}^{\otimes j-1})$, the class $\beta \otimes F_{\pi}$ comes from $H^{i-1}_{\acute{e}t}(R_{\pi}, \mu_{m}^{\otimes j-1})$, and as a result $\partial_{\pi}(\beta \cup (\delta)) = 0$. This means that $\partial_{\pi}(\widetilde{\zeta} - \beta \cup (\delta)) = 0$ Thus $\widetilde{\zeta} - \beta \cup (\delta)$ is unramified at all height one primes, by the exactness of the sequence in Theorem \ref{SAKAGAITO}, $\widetilde{\zeta} = \zeta_{0} + \beta \cup (\delta)$, where $\zeta_{0}$ is in $H^{i}_{\acute{e}t}(R, \mu_{m}^{\otimes j})$. Further, since $\widetilde{\zeta} \otimes F_{\pi} = 0$, it comes from $H^{i}(k(\pi), \mu_{m}^{\otimes j})$, and the reduction $\overline{\widetilde{\zeta}}$ is trivial. Since $k(\pi)$ is a complete discretely valued field with residue field $k$ and parameter $\overline{\delta}$, and since $\overline{\widetilde{\zeta}} = \overline{\zeta_{0}} + \overline{\beta}\cup (\overline{\delta}) = 0$, by the decomposition \eqref{SPLIT-DECOMP-WITT}, it follows that $\overline{\zeta_{0}} = 0$ and $\overline{\beta} = 0$. Finally since we have the isomorphism $H^{i}_{\acute{e}t}(R, \mu_{m}^{\otimes j}) \cong H^{i}(k, \mu_{m}^{\otimes j})$, we get that $\zeta_{0}$ and $\beta$ are trivial. Thus $\widetilde{\zeta} = 0$ and hence $\zeta = \alpha \cup(f)$.
\end{proof}
We are now in a position to prove the local-global principle for divisibility with respect to discrete valuations. We will first obtain the local-global principle with respect to points on the special fiber of any normal model of $F$. 

We will say that local-global principle for divisibility of a class $\zeta$ in $H^{i}(F, \mu_{m}^{\otimes j})$ by a class $\alpha$ in $H^{i-1}(F, \mu_{m}^{\otimes j-1})$ holds with respect to patches if the conclusion of Propositions \ref{LGP-WRT-PATCHES1} and \ref{LGP-WRT-PATCHES2} hold. 
\begin{prop}
\label{LGP-WRT-P}
Let $F$ be a semi-global field with residue field $k$, and let $\zeta$ be a class in $H^i(F, \mu_{m}^{\otimes j})$ and $\alpha$ be in $H^{i-1}(F, \mu_{m}^{\otimes j-1})$ for integers $i \geq 2, j \geq 1$ and $m$ coprime to $\Char(k)$. Suppose that local global principle for divisibility of $\zeta$ by $\alpha$ holds with respect to patches. If $\zeta = \alpha \cup (f_{P})$ for every point $P$ in the special fibre $\mathscr{X}_{k}$ of a normal projective model $\mathscr{X}$ of $F$, then  there exists an $f$ in $F^{\times}$ such that $\zeta = \alpha \cup (f)$.
\end{prop}
\begin{proof}
Let $\eta$ be a codimension one point of the special fibre $\mathscr{X}_{k}$. By hypothesis, $\zeta = \alpha \cup (f_{\eta})$ on the completion $F_{\eta}$ of $F$ for some $f_{\eta}$ in $F_{\eta}^{\times}$. Using weak approximation, one may choose an $f$ in $F^{\times}$ such that $(f) = (f_{\eta})$ in $H^{1}(F_{\eta}, \mu_{m})$ for every codimension one point $\eta$ of $\mathscr{X}_{k}$. Note that $\zeta - \alpha \cup (f) = 0$ in $H^i(F_{\eta}, \mu_{m}^{\otimes j})$. By \cite[Proposition 3.2.2]{HHK15(2)} there exists an affine open set $U \subseteq \overline{\{ \eta \}}$ which does not meet any other irreducible component of $\mathscr{X}_{k}$ such that $\zeta = \alpha \cup (f)$ in $H^{i}(F_U, \mu_{m}^{\otimes j})$. Let $\mathcal{P}$ be the complement in $\mathscr{X}$ of the union of these open sets $U$. By our assumption, for each $P$ in $\mathcal{P}$, there exist $f_{P}$ in $F_{P}^{\times}$ such that $\zeta = \alpha \cup (f_{P})$. Since $\zeta$ and $\alpha$ satisfy local global principle for divisibility with respect to patches, there exists an $f$ in $F^{\times}$ such that $\zeta = \alpha \cup (f)$.
\end{proof}

\begin{theorem}
\label{LGP-WRT-DVR}
Let $F$ be a semi-global field with residue field $k$, and let $\zeta$ be a class in $H^{i}(F, \mu_{m}^{\otimes j})$, and $\alpha$ be in $H^{i-1}(F, \mu_{m}^{\otimes j-1})$, for integers $i \geq 2, j \geq 1$ and $m$ coprime to $\Char(k)$. Suppose that local global principle for divisibility of $\zeta$ by $\alpha$ holds with respect to patches. If for every discrete valuation $v$ of $F$, there exist $f_{v}$ in the completion $F_{v}^{\times}$ such that $\zeta = \alpha \cup (f_{v})$, then there exists an $f$ in $F^{\times}$ such that $\zeta = \alpha \cup (f)$.
\end{theorem}
\begin{proof}
Let $\mathscr{X}/ \Sp{R}$ be a regular proper model of $F$. Using the embedded resolution of singularities (see \cite[Chapter 9, Theorem 2.26]{Li02}), we may assume that $\ram_{\mathscr{X}}(\zeta) \cup \ram_{\mathscr{X}}(\alpha)$ is a normal crossing divisor. Let $\{ v_1, \cdots v_n \}$ be the discrete valuations respectively corresponding to the irreducible components $\{ X_1, \cdots X_n \}$ of the special fiber $\mathscr{X}_{k}$. By assumption, there exist $f_{v_{i}}$ in $F_{v_{i}}^{\times}$ such that $\zeta = \alpha \cup (f_{v_{i}})$ in $H^{i}(F_{v_{i}}, \mu_{m}^{\otimes j})$. By weak approximation, there exists an $f$ in $F^{\times}$ such that $(f) = (f_{v_{i}})$ in $H^{1}(F_{v_{i}}, \mu_{m})$ for every $i = 1, \cdots, n$. Therefore we have that $\zeta - \alpha \cup (f) = 0$ in $H^{i}(F_{v_{i}}, \mu_{m}^{\otimes j})$ for every $i$. By \cite[Proposition 3.2.2]{HHK15(2)}, there exists an affine open subset $U_{i}$ of $X_i$ (not meeting any other irreducible component) such that $\zeta = \alpha \cup (f)$ on $H^{i}(F_{U_{i}}, \mu_{m}^{\otimes j})$ for every $i$. 

Let $\mathcal{P}$ be the remaining finite set of closed points in $\mathscr{X}_{k} \setminus \cup_{i=1}^{n} U_{i}$ including the points where distinct irreducible components of $\mathscr{X}_{k}$ meet, and at least one closed point in every irreducible component of $\mathscr{X}_{k}$. Let $P$ be in $\mathcal{P}$. Since $\ram_{\mathscr{X}}(\zeta) \cup \ram_{\mathscr{X}}(\alpha)$ is a normal crossing divisor, the ramification divisor of $\zeta$ and $\alpha$ (viewed as classes in $F_{P}$) $\ram_{\widehat{R_{P}}}(\zeta) \cup \ram_{\widehat{R_{P}}}(\alpha)$ is supported along a regular system of parameters $\{ \pi, \delta \}$ of $\widehat{R_{P}}$. Let $\mathfrak{p} := (\pi)\cap \mathcal{O}_{\mathscr{X}, P}$ be the contraction of the prime ideal $(\pi)$. The prime ideal $\mathfrak{p}$ gives a discrete valuation $v$ of $F$. Note that the completion $F_{v}$ is contained in the completion $F_{P, \pi}$ of $F_P$ at $\pi$. Since $\zeta = \alpha \cup (f_{v})$ for some $f_v$ in $F_{v} \subset F_{P, \pi}$, we have that $\zeta = \alpha \cup (f_{v})$ also in $H^{i}(F_{P, \pi}, \mu_{m}^{\otimes j})$. Because $\zeta$ and $\alpha$ are unramified everywhere except possibly at $\pi$ and $\delta$, by Proposition \ref{TWO-DIM-PROP}, there exists an $f_{P}$ in $F_{P}^{\times}$ such that $\zeta = \alpha \cup (f_{P})$ on $H^{i}(F_{P}, \mu_{m}^{\otimes j})$ for every $P$ in $\mathcal{P}$. Therefore by Proposition \ref{LGP-WRT-P}, we conclude that there exists an $f$ in $F^{\times}$ such that $\zeta = \alpha \cup (f)$.
\end{proof}

\begin{corollary}
\label{DIVISIBILITY}
Let $F$ be a semi-global field, and let $\zeta$ be a class in $H^{i}(F, \mu_{m}^{\otimes j})$ and a non-zero class $\alpha$ in $H^{i-1}(F, \mu_{m}^{\otimes j-1})$, where $i, j >0$ and $m$ is an integer coprime to $\Char(k)$. Suppose that one of the following condition holds:
\begin{enumerate}
\item $i=3$, $j=2$ and every branch field $F_{\wp}$ of $F$ is Rost $m$-divisible.
\item $i = 3$, $j = 2$ and $\ind(\alpha)$ is square-free.
\item $i \geq 2$, $m = 2$, and $\alpha$ is a symbol. 
\end{enumerate}
If for every discrete valuation $v$, there exists an $f_{v}$ in $F_{v}^{\times}$ such that $\zeta = \alpha \cup (f_{v})$ on $H^{i}(F_{v}, \mu_{m}^{\otimes j})$, then there exists an $f$ in $F^{\times}$ such that $\zeta = \alpha  \cup (f)$.
\end{corollary}
\begin{proof}
By Theorem \ref{LGP-WRT-DVR}, it is enough to show that $\zeta$ and $\alpha$ satisfy local global principle with respect to patches. If $i =3$, $j =2$ and $F_{\wp}$ is Rost $m$-divisible or $\ind(\alpha)$ is square-free, this follows from Proposition \ref{LGP-WRT-PATCHES1}. If $m =2$ and $\alpha$ is a symbol, this follows from Proposition \ref{LGP-WRT-PATCHES2}.
\end{proof}

\section{Proof of the main theorem}
Before we prove Theorem \ref{MAIN-THM2}, we prove the following lemma, most probably well known to experts.
\begin{lemma}
	\label{WELL-KNOWN-LEMMA}
	Let $R$ be a complete two dimensional regular local ring with fraction field $F$ and residue field $k$ and maximal ideal $\mathfrak{m} = (\pi, \delta)$. Suppose that $m > 1$ is an integer coprime to $\Char{k}$. For every prime $\ell$ dividing $m$ suppose that $\cd_{\ell}(k) \leq i-2$.  If $\zeta$ is a class in $H^{i}(F, \mu_{m}^{\otimes j})$ which is possibly ramified only along $\pi$, then in fact $\zeta = 0$.
\end{lemma}
\begin{proof}
	Note that by the proper base change theorem (see \cite[Chapter VI, Corollary 2.7]{Mi80}), we have $H^{i}_{\acute{e}t}(R, \mu_{m}^{\otimes j}) \cong H^{i}(k, \mu_{m}^{\otimes j})$. Since $\cd_{\ell}(k) \leq i -2$ for every prime $\ell$ dividing $m$, it follws that $H^{i}(k, \mu_{m}^{\otimes j}) = 0$. By Theorem \ref{SAKAGAITO}, suffices to show that $\zeta$ is unramified at every height one prime of $R$. By the hypothesis, we only need to check that $\zeta$ is unramified along $(\pi)$. By the exact sequence in Theorem \ref{SAKAGAITO}, we have $\partial_{\overline{\delta}}(\partial_{\pi}(\zeta)) = - \partial_{\overline{\pi}}(\partial_{\delta}(\zeta)) = 0$. Note that the residue field $k(\pi)$ is a complete discretely valued field with parameter the reduction $\overline{\delta}$ and residue field $k$. The class $\partial_{\pi}(\zeta)$ lies in $H^{i-1}(k(\pi), \mu_{m}^{\otimes j-1})$ and is unramified. Therefore by the Witt exact sequence (see \eqref{WITT-EXACT-SEQUENCE}), $\partial_{\pi}(\zeta)$ comes from $H^{i-1}(k ,\mu_{m}^{\otimes j-1})$. But since $\cd_{\ell}(k) \leq i-2$ for every $\ell$ dividing $m$, we also have $H^{i-1}(k, \mu_{m}^{\otimes j-1}) = 0$, and hence $\partial_{\pi}(\zeta) = 0$. This shows that $\zeta = 0$.
\end{proof}

\begin{proof}[\textbf{Proof of Theorem \ref{MAIN-THM2}}]:
	Note that every branch field $F_{\wp}$ of every normal model of $F$ is a complete discretely valued field with residue field a (positive characteristic) local field. In view of \cite[Theorem 4.12]{PPS18}, $F_{\wp}$ is Rost $m$-divisible. By Corollary \ref{DIVISIBILITY}, for every $\zeta$ in $H^{3}(F, \mu_{m}^{\otimes 2})$, and for every discrete valuation $v$ of $F$, it suffices to show that there exists a symbol algebra $\alpha$ in $H^{2}(F, \mu_{m})$ such that $\zeta = \alpha \cup (g_{v})$ in $H^{3}(F_{v}, \mu_{m}^{\otimes 2})$. 
	
	Let $\mathscr{X}/\Sp{R}$ be a regular projective model of $F$ such that $\ram_{\mathscr{X}}(\zeta)$ is a normal crossing divisor in $\mathscr{X}$ (see \cite[Chapter 9, Theorem 2.26]{Li02}). By blowing up if necessary, we may also assume that each irreducible component of (the support of) this divisor is a regular curve. Let $\{ C_{1}, \cdots, C_{n} \}$ be the irreducible components of $\Supp(\ram_{\mathscr{X}}(\zeta))$.
	
	Let $\mathcal{P}$ be the set of self intersection points of $\Supp(\ram_{\mathscr{X}}(\zeta))$. Since $\mathscr{X}$ is projective (in particular, quasi-projective), there exists an affine open set $\Sp(A)$ containing all points in $\mathcal{P}$ including at least one point in every connected component of $\Supp(\ram_{\mathscr{X}}(\zeta))$. Semi-localizing $A$ at the points in $\mathcal{P}$, we obtain a semi-local ring $B$. Since $\Pic(B)$ is trivial, the divisors $C_{i}$, when restricted to $\Sp(B)$ are given by some functions $\pi_{i}$ in $F$ respectively. As a result, $\Div_{\mathscr{X}}(\pi_{i}) = C_{i} + E_{i}$, where $E_{i}$ does not pass through any point in $\mathcal{P}$. 
	
	Note that the function fields $k(C_{i})$ of $C_{i}$ are either (positive characteristic) global fields or non-archimedean local fields.
	Let $P$ be a point in $\mathcal{P} \cap C_{i}$.  Note that $\chi_{i, P}: = \partial_{P}(\partial_{C_{i}}(\zeta))$ is the residue of the Brauer class $\partial_{C_{i}}(\zeta)$ at the point $P$. Since $m$ is coprime to $\Char{k}$, for the largest $2^{n}$ dividing $m$, the field extension $k(C_{i})(\mu_{2^{n}})/k(C_{i})$ is cyclic. By the Grunwald-Wang theorem (see for eg. \cite[Theorem 1]{LR99}), there exists $\chi_{i}$ in $H^{1}(k(C_{i}), \mathbb{Z}/m\mathbb{Z})$ such that $\chi_{i}$ induces $\chi_{i, P}$ at every $P$ in $\mathcal{P}\cap C_{i}$. Observe that since $\ram_{\mathscr{X}}(\zeta)$ is a normal crossing divisor, $P$ lies on at most one another irreducible component $C_{j}$.
	
	If $P$ also lies on $C_{j}$, then by the exact sequence in Theorem \ref{SAKAGAITO}, we have $\partial_{P}(\partial_{C_{i}}(\zeta)) + \partial_{P}(\partial_{C_{j}}(\zeta)) = 0$. Therefore $\chi_{j, P}: = \partial_{P}(\partial_{C_{j}}(\zeta)) = - \chi_{i, P}$. If $P$ is not on any other $C_{j}$, then $\partial_{P}(\partial_{C_{i}}(\zeta)) = 0$. Thus the Brauer class $\partial_{C_{i}}(\zeta)$ is only possibly ramified at the points in $\mathcal{P} \cap C_{i}$. Let $L_{i}/k(C_{i})$ be the extension given by the character $\chi_{i}$. Note that since $L_{i}/k(C_{i})$ splits the residues of $\partial_{C_{i}}(\zeta)$, we have that $\partial_{C_{i}}(\zeta) \otimes_{k(C_{i})}L_{i} = 0$, and thus $\partial_{C_{i}}(\zeta) = (\chi_{i}, a_{i})$ for some $a_{i}$ in $k(C_{i})^{\times}$.
	
	Let $v_{i}$ be the discrete valuation corresponding to $C_{i}$. Observe that the parameters at the completions $F_{v_{i}}$ are given by $\pi_{i}$. Note that $\chi_{i}$ can be thought of as elements in $H^{1}(F_{v_{i}}, \mathbb{Z}/m\mathbb{Z})$. Again by the Grunwald-Wang theorem (\cite[Theorem 1]{LR99}), there exists $\chi$ in $H^{1}(F, \mathbb{Z}/m\mathbb{Z})$ such that $\chi$ induces $\chi_{i}$ at $F_{v_{i}}$. 
	
	We set $f: = \prod_{i=1}^{n}\pi_{i}$ and set $\alpha: = (\chi, f)$.
	We claim that there exists a $g$ such that $\zeta = \alpha \cup (g)$. As noted before by Corollary \ref{DIVISIBILITY}, it suffices to show that for every discrete valuation $v$, we have $\zeta \otimes F_{v} = (\chi, f) \cup (g_{v})$ for some $g_{v}$ in $F_{v}$.
	
	Let $v$ be a discrete valuation of $F$. Suppose that the center of $v$ on $\mathscr{X}$ is a point $P$ outside $\ram_{\mathscr{X}}(\zeta)$. Note that $F_{P} \subseteq F_{v}$. Since $\zeta \otimes F_{P}$ is unramified with respect to every height one prime ideal of $\widehat{R_{P}}$, and since $\cd_{\ell}(k(P)) \leq 2$ for every prime $\ell$ dividing $m$, by Theorem \ref{SAKAGAITO}, $\zeta \otimes F_{P} = 0$ and therefore $\zeta \otimes F_{v} = 0$. 
	
	Now suppose that the center $P$ of $v$ on $\mathscr{X}$ is the generic point of $C_{i}$. Note that $\partial_{v}((\chi, f) \cup (a_{i}^{-1})) = \partial_{C_{i}}((\chi_{i}, w\pi_{i})\cup (a_{i}^{-1})) $, where $w$ is a unit in $F_{v}^{\times}$. We also abusively denote the lift of $(a_{i})$ to $F_{v}$ by $(a_{i})$. Further we have, 
	\begin{align*}
	\partial_{C_{i}}((\chi_{i}, w\pi_{i})\cup (a_{i}^{-1})) &= \partial_{C_{i}}((\chi_{i}, a_{i}) \cup (w\pi_{i})) \\
	              &= (\chi_{i}, a_{i}) \\
	              &= \partial_{C_{i}}(\zeta).
	\end{align*}
    Thus we see that $\partial_{C_{i}}(\zeta \otimes F_{v} - (\chi, f) \cup (a_{i}^{-1})) = 0$. Since $\cd(k(C_{i})) \leq 2$, by the Witt Exact sequence (see \eqref{WITT-EXACT-SEQUENCE}, Subsection \ref{GALOIS-COHOMOLOGY}), it follows that $\zeta \otimes F_{v} = (\chi, f) \cup (a_{i}^{-1})$.
    
    Now suppose that the center of $v$ is a codimension two point $P$ in $\ram_{\mathscr{X}}(\zeta)$. Suppose that $P$ lies in exactly one $C_{i}$. We claim that $\zeta \otimes F_{P} = 0$. Since $P$ does not lie in any other $C_{j}$, we have $\partial_{P}(\partial_{C_{i}}(\zeta)) = 0$. Let $\mathfrak{p}$ be a height one prime ideal in $\widehat{R_{P}}$ defining $C_{i}$. Note that $\partial_{P}(\partial_{\mathfrak{p}}(\zeta \otimes F_{P})) = 0$ and $\zeta \otimes F_{P}$ is unramified at every other height one prime ideal. Thus by Lemma \ref{WELL-KNOWN-LEMMA}, we obtain that $\zeta \otimes F_{P} = 0$. Since $F_{P} \subset F_{v}$, we have $\zeta \otimes F_{v} = 0$ and therefore $\zeta \otimes F_{v} = (\chi, f) \cup (1)$.  
    
    Finally suppose that the center of $v$ is a codimension two point $P$ in $C_{i} \cap C_{j}$. Note that on $F_{P}$, we have $(\chi, f) \cup (\pi_{i}) = (\chi, w\pi_{i}\pi_{j}) \cup (\pi_{i})$, where $w$ is a unit in $\widehat{R_{P}}$.
    \begin{align*}
    	(\chi, f) \cup (\pi_{i}) &= (\chi, w \pi_{i}\pi_{j}) \cup (\pi_{i}) \\
    							&= (\chi, -w) \cup (\pi_{i}) + (\chi, \pi_{j}) \cup (\pi_{i}). 
    \end{align*}
Since $\chi$ is unramified on $F_{P}$, and $-w$ is a unit in $\widehat{R_{P}}$, $(\chi, - w) = 0$. Thus on $F_{P}$, we have $(\chi, f) \cup (\pi_{i}) = (\chi, \pi_{j}) \cup (\pi_{i})$. We set $\zeta^{\prime} : = (\chi, \pi_{j})\cup (\pi_{i})$. Note that $\partial_{\overline{\pi_{j}}}(\partial_{\pi_{i}}(\zeta^{\prime} \otimes F_{P})) = \partial_{\overline{\pi_{j}}}(\chi_{i}, \overline{\pi_{j}}) = \chi_{i, P} $. Thus $\partial_{\overline{\pi_{j}}}(\partial_{\pi_{i}}(\zeta \otimes F_{P})) = \partial_{\overline{\pi_{j}}}(\partial_{\pi_{i}}(\zeta^{\prime} \otimes F_{P}))$. Since the residue field $k(\pi_{i})$ of $\widehat{R_{P}}$ at $\pi_{i}$ is a non-archimedean local field, $\partial_{\pi_{i}}(\zeta \otimes F_{P}) = \partial_{\pi_{i}}(\zeta^{\prime} \otimes F_{P})$. Similarly one sees that $\partial_{\pi_{j}}(\zeta \otimes F_{P}) = \partial_{\pi_{j}}(\zeta^{\prime} \otimes F_{P})$. Since $\zeta$ is unramified on $\widehat{R_{P}}$ at any other height one prime ideal $\mathfrak{p}$ of $\widehat{R_{P}}$, we have $\partial_{\mathfrak{p}}(\zeta^{\prime}) = 0$. Thus $\zeta \otimes F_{P} - \zeta^{\prime} \otimes F_{P}$ is unramifed at every height one prime ideal of $\widehat{R_{P}}$ and hence $\zeta \otimes F_{P} = \zeta^{\prime} \otimes F_{P} = (\chi, f) \cup (\pi_{i})$. Since $F_{P} \subseteq F_{v}$, we obtain that $\zeta \otimes F_{v} = (\chi, f) \cup (\pi_{i})$. This finishes the proof.	
\end{proof}

\bibliographystyle{amsart}

\end{document}